\documentclass[10pt]{amsart}

\usepackage{amsfonts}
\usepackage{amssymb}
\usepackage{amsmath, amsthm, latexsym}
\usepackage[all]{xy}
\usepackage{hyperref}
\usepackage[margin=1.5in]{geometry}
\usepackage{color}

\def \C{\mathbb{C}}
\def \Z{\mathbb{Z}}
\def \R{\mathbb{R}}

\def \GL{\operatorname{GL}}

\def \Sp{\operatorname{Sp}}
\def \SO{\operatorname{SO}}
\def \Aut{\operatorname{Aut}}
\def \k{\textbf{k}}
\def \B{\mathfrak{B}}

\newcommand*\abs[1]{\left\lvert#1\right\rvert}%absolute value
\newcommand*\dotpro[2]{\langle{#1} , {#2}\rangle}%inner product

\theoremstyle{plain}
\newtheorem{theorem}{Theorem}[section]
\newtheorem{lemma}[theorem]{Lemma}
\newtheorem{prop}[theorem]{Proposition}
\newtheorem{corollary}[theorem]{Corollary}

\theoremstyle{definition}
\newtheorem{example}[theorem]{Example}
\newtheorem{definition}[theorem]{Definition}
\newtheorem{remark}[theorem]{Remark}
\newtheorem{problem}[theorem]{Problem}

\begin{document}

\title{Toric principal bundles, Tits buildings and reduction of structure group}

\author{Shaoyu Huang}
\address{Department of Mathematics, University of Pittsburgh,
Pittsburgh, PA, USA.}
\email{shh123@pitt.edu}

\author{Kiumars Kaveh}
\address{Department of Mathematics, University of Pittsburgh, Pittsburgh, PA, USA.}
\email{kaveh@pitt.edu}

%\begin{abstract}
%\end{abstract}

%\thanks{The first author is partially supported by a
%Simons Foundation Collaboration Grants for Mathematicians (Grant ID: 210099) and a National Science Foundation Grant 
%(Grant ID: 1200581).}

%\keywords{} 
%\subjclass[2010]{Primary:}
\date{\today}

\begin{abstract}
A toric principal $G$-bundle is a principal $G$-bundle over a toric variety together with a torus action commuting with the $G$-action. In \cite{Kaveh-Manon}, extending the Klyachko classification of toric vector bundles, toric principal bundles are classified using \emph{piecewise linear maps} to the (extended) Tits building of $G$. In this paper, we use the classification in \cite{Kaveh-Manon} to give a description of the (equivariant) automorphism group of a toric principal bundle as well as a simple criterion for (equivariant) reduction of structure group, recovering results of Dasgupta et al in \cite{Dasgupta}. Finally, motivated by the equivariant splitting problem for toric principal bundles, we introduce the notion of \emph{Helly's number} of a building and pose the problem of giving sharp upper bounds for Helly's number of Tits buildings of semisimple algebraic groups $G$.    
\end{abstract}

\maketitle
\tableofcontents

\section{Introduction}
Throughout $\k$ denotes an algebraically closed field. Let $G$ be a linear algebraic group over $\k$. Also let $T$ denote a torus over $\k$ and $X_\Sigma$ the $T$-toric variety corresponding to a fan $\Sigma$. A \emph{toric principal $G$-bundle} $\mathcal{P}$ over $X_\Sigma$ is a principal $G$-bundle over $X_\Sigma$ together with a $T$-action on $\mathcal{P}$, lifting its action on $X_\Sigma$, such that the $T$-action and the $G$-action commute. The (isomorphism classes of) toric principal $\GL(r)$-bundles are in one-to-one correspondence with the (isomorphism classes of) rank $r$ toric vector bundles. %In \cite{Biswas} toric principal bundles are classified with certain data of cocycles and homomorphisms.

Throughout the paper we fix a point $x_0$ in the open torus orbit in $X_\Sigma$. It gives an identification of the open orbit with the torus $T$. 
By a \emph{framed} toric principal bundle $(\mathcal{P}, p_0)$ we mean a toric principal bundle $\mathcal{P}$ together with the choice of a point $p_0$ in the fiber $\mathcal{P}_{x_0}$. This choice gives an identification of the fiber $\mathcal{P}_{x_0}$ with the group $G$. 

When $G$ is a reductive algebraic group, in \cite{Kaveh-Manon}, the authors give a classification of framed toric principal $G$-bundles on $X_\Sigma$ in terms of \emph{piecewise linear maps} from $|\Sigma|$, the support of $\Sigma$, to $\tilde{\B}(G)$, \emph{the cone over the Tits building of $G$} (see \cite[Theorem 2.2 and Theorem 2.4]{Kaveh-Manon}). This classification is in the spirit of the Klyachko classification of toric vector bundles (\cite{Klyachko}). 
In \cite{Biswas} toric principal bundles are classified with certain data of cocycles and homomorphisms. The classification in \cite{Biswas} is in the spirit of the Kaneyama classification of toric vector bundles (\cite{Kaneyama}).

In \cite{Dasgupta}, the authors use the Kaneyama type classification in \cite{Biswas}, to obtain interesting results on (equivariant) automorphism group, (equivariant) reduction of structure group and stability of toric principal bundles.  
In the present short paper, we use the classification in \cite{Kaveh-Manon} to give short proofs for some of the results in \cite{Dasgupta}. The following is a more specific description of the content of the paper:
\begin{itemize}
\item We give a short proof of a result of Dasgupta et al (\cite[Proposition 5.1]{Dasgupta}) describing the equivariant automorphism group of a (non-framed) toric principal bundle, as an intersection of certain parabolic subgroups of $G$ (Theorem \ref{th-auto-gp}).
\item We give a simple criterion for the equivariant reduction of structure group of a toric principal bundle (see Definition \ref{def-eq-str-gp} and Theorem \ref{th-cri-red-st-gp}). More precisely, we show that a toric principal $G$-bundle has an equivariant reduction of structure group to a closed subgroup $K$, if and only if, for some choice of a frame $p_0$, the image of the corresponding piecewise linear map $\Phi: |\Sigma| \to \tilde{\B}(G)$ lies in $\tilde{\B}(K)$.  
As corollaries we recover the results in \cite{Dasgupta} regarding reduction of structure group and splitting of toric principal bundles (see Corollary \ref{cor-eq-split}, Corollary \ref{cor-split-K-G}). 
\item We introduce the notion of \emph{Helly's number} of a building and pose the problem of finding sharp upper bounds for it (see Definition \ref{def-helly-gp}). For the Tits building of a group $G$, this Helly's number is directly related to the problem of splitting of toric principal $G$-bundles over projective spaces (Corollary \ref{cor-proj-split}).   
\end{itemize}

\bigskip
\noindent{\bf Acknowledgement:} The second author is supported by the National Science Foundation grant (DMS-2101843). We thank Roman Fedorov, Mainak Poddar, Jyoti Dasgupta, Bivas Khan, Christopher Manon and Sam Payne for useful discussions and comments.

\section{Preliminaries on Tits buildings}
\subsection{Tits building of a linear algebraic group}
In this section we review some basic facts that we need about the Tits buildings associated to linear algebraic groups. 

A \emph{building} is a pair $(\Delta,\mathcal{A})$ consisting of a simplicial complex $\Delta$ and a family $\mathcal{A}$ of subcomplexes $A$ (\emph{apartments}) satisfying certain conditions. Readers can find the general definition of building in the appendix (Definition \ref{def-bldg}).

To a linear algebraic group $G$ over a field $\k$ there corresponds a building called \emph{Tits building} of $G$. We denote it by $\Delta(G)$. The set of simplices in $\Delta(G)$ is the set of parabolic subgroups of $G$ ordered by reverse inclusion. The apartments in $\Delta(G)$ correspond to maximal tori in $G$. For a maximal torus $H \subset G$, the corresponding apartment consists of parabolic subgroups containing $H$. Clearly, Borel subgroups correspond to the maximal simplices, i.e. chambers, in $\Delta(G)$. Since every parabolic subgroup contains the solvable radical $R(G)$ of $G$, $\Delta(G)$ and $\Delta(G/R(G))$ are isomorphic as simplicial complexes. 
%are the same.

\begin{example}[Tits building of $\GL(r)$]
Consider $G = \GL(r)$. Any parabolic subgroup $P$ in $\GL(r)$ is the stabilizer of a flag $F_\bullet=(\{0\}=F_0 \subsetneqq F_1\subsetneqq \cdots \subsetneqq F_k=\C^r)$. This gives a one-to-one correspondence between the simplices in the Tits building of $\GL(r)$ and flags in $\C^r$. In particular, Borel subgroups are stabilizers of complete flags and correspond to chambers in $\Delta(\GL(r))$. A \emph{frame} $L$ in $\C^r$ is a direct sum decomposition of $\C^r = \bigoplus_{i=1}^r L_i$ into one-dimensional subspaces $L_i$. In other words, a \emph{frame} is an equivalence class of vector spaces bases up to scaling basis elements by non-zero scalars. We say that a flag $F_\bullet$ is \emph{adapted to a frame $L$} if each subspaces $F_i$ is spanned by some of the $L_j$. 
The apartments in the Tits building of $\GL(r)$ correspond to frames in $\C^r$. The apartment corresponding to a frame $L$ consists of all the flags adapted to it. 
\end{example}

\begin{example}[Tits building of $\Sp(2r)$]
Consider $G=\Sp(2r)\subset \GL(2r)$. We denote by $\langle \cdot, \cdot \rangle$ the standard skew symmetric bilinear form $\sum_i x_i \wedge y_i$ on $\C^{2r}$. We call a flag $F_\bullet = \left( \{0\}=F_0 \subsetneqq F_1\subsetneqq \cdots \subsetneqq F_k=\C^{2r}\right)$, an \emph{isotropic flag} if for each $0\le j\le k$ we have
    \begin{equation*}
        F_j^\perp = F_{k-j}.
    \end{equation*}
Any parabolic subgroup of $\Sp(2r)$ is the stabilizer of an isotropic flag. 
%    \begin{equation}
%        F_\bullet = \left( \{0\} \subsetneqq F_1\subsetneqq \cdots \subsetneqq F_k=\C^{2r}\right)
%    \end{equation}
%    is a partial flag of subspaces.

We say that a basis $B=\{e_1,\cdots, e_r, f_1,\cdots, f_r\}$ for $\C^{2r}$ is a \emph{normal basis} if the following holds:
\begin{equation*}
    \begin{split}
        \dotpro{e_i}{e_j} &= 0,\ \forall i, j\\
        \dotpro{f_i}{f_j} &= 0,\ \forall i, j\\
        \dotpro{e_i}{f_i} &= 1,\ \forall i\\
        \dotpro{e_i}{f_j} &= 0,\ \forall i, j, i\neq j.
    \end{split}
\end{equation*}
One knows that normal bases exist. If $B=\{e_1,\cdots, e_r, f_1,\cdots, f_r\}$ is a normal basis then $\{t_1e_1,\cdots, t_re_r, t_1^{-1}f_1,\cdots, t_r^{-1}f_r\}$, for any non-zero $t_1,\cdots,t_r$, is also a normal basis. We call a normal basis, up to multiplication by non-zero scalars $t_i$, a \emph{normal frame}. The normal frames are in one-to-one correspondence with maximal tori of $G$ and hence with apartments in $\Delta(G)$. The apartment corresponding to a normal frame $L$ consists of all the isotropic flags that are adapted to $L$. 
\end{example}

When $G$ is semisimple, the simplicial complex $\Delta(G)$ has a natural geometric realization. Namely, there is a topological space $\B(G)$ together with a triangulation in which simplices in the triangulation (which are subsets of $\B(G)$ homeomorphic to standard simplices) are in one-to-one correspondence with the simplices in $\Delta(G)$ and intersect according to how simplices in $\Delta(G)$ intersect. It is constructed as follows. For each maximal torus $H \subset G$ let $\Lambda^\vee(H)$ be its cocharacter lattice and let $\Lambda^\vee_\R(H)= \Lambda^\vee(H) \otimes_\Z \R$. The apartment corresponding to $H$ is the triangulation of the unit sphere in 
$\Lambda^\vee_\R(H)$ obtained by intersecting it with the Weyl chambers and their faces. Two simplices, in different apartments, are glued together if the corresponding faces represent the same parabolic subgroup in $G$. 

\begin{definition}[Geometric realization of the Tits building]  \label{def-B(G)}
The topological space $\B(G)$ is obtained by gluing the unit spheres in the $\Lambda^\vee_\R(H)$, for all maximal tori $H$, along their common simplices. 
\end{definition}

While in our notation, we distinguish between the building as an abstract simplicial complex, i.e. $\Delta(G)$, and as a topological space, i.e. $\B(G)$, by abuse of terminology we refer to both $\Delta(G)$ and $\B(G)$ as the Tits building of $G$. 

\begin{definition}[Cone over the Tits building of a semisimple group]   \label{def-tilde-B(G)}
Let $G$ be semisimple. Similar to the construction of $\B(G)$, we  construct the topological space $\tilde{\B}(G)$ by gluing the vector spaces $\Lambda^\vee_\R(H)$, along their common faces of Weyl chambers. We think of $\tilde{\B}(G)$ as the cone over $\B(G)$ and call it the \emph{cone over the Tits building of $G$}.
\end{definition}

%Let $H$ be a maximal torus in $G$ with cocharacter lattice $\Lambda^\vee(H)$. %denotes the cocharacter lattice of $H$. 
%There is a one-to-one correspondence between
%the faces of Weyl chambers in $\Lambda^\vee_\R(H)=\Lambda^\vee(H)\otimes_\Z\R$ and the parabolic subgroups containing $H$. We can glue the real vector spaces $\Lambda_\R^\vee(H)$ (respectively the lattices $\Lambda^\vee(H)$), for all maximal tori $H$,
%along the faces corresponding to the same parabolic subgroups. We denote the resulting space by $\tilde{\mathfrak{B}}(G)$ (respectively $\tilde{\mathfrak{B}}_\Z(G)$), the \emph{cone over the Tits building} of $G$ (respectively the set of \emph{lattice points in the cone over the Tits building}).

Now let $G$ be a linear algebraic group and let $G_{\textup{ss}} = G / R(G)$ be the semisimple quotient of $G$. The previous construction in the semisimple case works in this case as well and we can define $\tilde{\B}(G)$ (respectively $\B(G)$) to be the topological space obtained by gluing the vector spaces $\Lambda^\vee_\R(H)$ (respectively unit spheres in the $\Lambda^\vee_\R(H)$), for all maximal tori $H \subset G$, along their common faces of Weyl chambers (respectively intersections of common faces with the unit spheres). 
When $G$ is reductive, the topological space $\tilde{\B}(G)$ is the Cartesian product of $\tilde{\B}(G_{\textup{ss}})$ with the real vector space $\Lambda^\vee(Z) \otimes_\Z \R$, where $Z=Z(G)^\circ$ is the connected component of the identity in the center of $G$. 

\begin{definition}[Extended Tits building of a linear algebraic group]      \label{def-ext-bldg}
For a linear algebraic group $G$, we refer to $\tilde{\B}(G)$ (above) as the \emph{extended Tits building} of $G$. Also, for a maximal torus $H$, we refer to $\Lambda^\vee_\R(H)$ as the \textit{cone} over the apartment of $H$ and denote it by $\tilde{A}_H$. When $G$ is semisimple, the extended Tits building $\tilde{\B}(G)$ is the cone over the Tits building of $G$.

We denote by $\tilde{\B}_\Z(G)$ the subset of $\tilde{\B}(G)$ obtained by gluing the lattices $\Lambda^\vee(H)$, for all maximal tori $H$,
and call it the set of \emph{lattice points in the extended Tits building of $G$}.
\end{definition}

\begin{remark}
Our choice of terminology \emph{an extended Tits building} is motivated by a similar term, namely  \emph{an extended Bruhat-Tits building}, from the theory of Bruhat-Tits buildings for algebraic groups over valued fields (see \cite{Tits} as well as \cite[Remark 1.23]{RTW}).
\end{remark}

We will see in Section \ref{subsec-1-para-subgp} that the set $\tilde{\B}_\Z(G)$ of lattice points in $\tilde{\B}(G)$ can be identified with the set of one-parameter subgroups of $G$ modulo certain equivalence relation (Definition \ref{def-1-para-equiv} and Proposition \ref{prop-tilde-B-one-para}).

%\begin{definition}
%    Let $G$ be a linear algebraic group. For any maximal torus $H\subset G$, let $A_H=\Lambda^\vee(H)\otimes_\Z\R\setminus\{0\}/\R_{>0}$, where $\Lambda^\vee(H)$ is the cocharacter lattice of $H$. Consider the
%    topological space $\mathfrak{B}(G)$ obtained by gluing $A_H$, for all maximal tori $H\subset G$, along faces corresponding to the same parabolic subgroups. We call $\mathfrak{B}(G)$ the \emph{underlying space of the Tits building of $G$}.
%\end{definition}

%\begin{remark}
%    When $G$ is semisimple. $\Delta(G)$ can be identified with $\mathfrak{B}(G)$ in the sense that each simplex in $\Delta(G)$ can be identified with a subset
%    of $\mathfrak{B}(G)$ that is homeomorphic to a standard simplex.
%\end{remark}

\subsection{One-parameter subgroups and Tits building} 
\label{subsec-1-para-subgp}
In this section, following \cite[Section 1.3]{Kaveh-Manon}, we present a natural way to realize the extended Tits building of $G$ in terms of one-parameter subgroups of $G$. More precisely, we see that the set of lattice points $\tilde{\B}_\Z(G)$ in $\tilde{\B}(G)$ can naturally be identified with certain equivalence classes of one-parameter subgroups in $G$ (Proposition \ref{prop-tilde-B-one-para}). For details and proofs we refer the reader to \cite[Section 1.3]{Kaveh-Manon}. This construction of the Tits building of a linear algebraic group from one-parameter subgroups also appears, in slightly different form, in \cite[Section 2.2]{Mumford}.

\begin{definition}   \label{def-1-para-equiv}
Let $\lambda_1,\lambda_2$ be algebraic one-parameter subgroups of $G$. We say that $\lambda_1$ is \emph{equivalent} to $\lambda_2$ and write $\lambda_1\sim\lambda_2$ if $\displaystyle\lim_{s\to0}\lambda_1(s)\lambda_2(s)^{-1}$ exists in $G$.
\end{definition}

It is easy to see this is indeed an equivalence relation. 

\begin{definition}[Parabolic subgroup associated to a one-parameter subgroup]
For a one-parameter subgroup $\lambda:\mathbb{G}_m\to G$, let
    \begin{equation*}
        P_\lambda = \{g\in G\mid \lim_{s\to 0}\lambda(s)g\lambda(s)^{-1}\text{ exists in } G\}.
    \end{equation*}
One shows that $P_\lambda$ is a parabolic subgroup in $G$. 
\end{definition}

Alternatively, $P_\lambda$ can be described in terms of the equivalence relation $\sim$ (see \cite[Proposition 1.8]{Kaveh-Manon}): 
\begin{equation}  \label{equ-P-lambda}
P_\lambda = \{ g \in G \mid g \lambda g^{-1} \sim \lambda\}.
\end{equation}
It is straightforward to check that if $\lambda_1 \sim \lambda_2$ then $P_{\lambda_1} = P_{\lambda_2}$. Thus to each equivalence class of one-parameter subgroups there corresponds a parabolic subgroup. One also shows that, for a maximal torus $H \subset G$, no two one-parameter subgroups in $\Lambda^\vee(H)$ are equivalent. Moreover, if a one-parameter subgroup $\lambda \in \Lambda^\vee(H)$ lies in the relative interior of a face of a Weyl chamber, the parabolic subgroup $P_\lambda$ is exactly the parabolic subgroup corresponding to this face. Putting these facts together one obtains the following (\cite[Corollary 1.11]{Kaveh-Manon}).

\begin{prop}  \label{prop-tilde-B-one-para}
The set $\tilde{\mathfrak{B}}_\Z(G)$ can naturally be identified with the set of equivalence classes of one-parameter subgroups of $G$.
\end{prop}

\begin{remark}
The above realization of the (extended) Tits building of $G$ in terms of equivalence classes of one-parameter subgroups (Proposition \ref{prop-tilde-B-one-para}) is analogous to the description of the Tits building of a symmetric space as the set of equivalence classes of geodesics (see \cite[Section 3]{Ji}).
\end{remark}

\begin{example}
Consider $G=\text{Sp}(2r)$. A one-parameter subgroup $\lambda:\mathbb{G}_m\to G$ is given by a diagonal matrix $$\text{diag}(t^{v_1},\cdots,t^{v_r},t^{-v_r},\cdots,t^{-v_1}),\ v_i\in\mathbb{Z}$$ under some ordered normal basis $\{e_1,\cdots,e_r, f_r,\cdots, f_1\}$. After reordering and switching, we may assume $v_1\ge\cdots\ge v_r\ge 0\ge-v_r\cdots\ge -v_1$, which will still gives us a normal basis. For
$i=1,\cdots, r$, let $v_{r+i}=-v_{r+1-i}$. Consider indices $i_1,\cdots,i_k=2r$ such that $v_1=\cdots=v_{i_1}>v_{i_1+1}=\cdots=v_{i_2}>\cdots>v_{i_{k-1}+1}=\cdots=v_{i_k}=v_{2r}$. For $j=1,\cdots,k$, we let $c_j=v_{i_j}$ and $F_j=V_{i_j}$ which is spanned by first $i_j$ vectors in ordered normal basis. Then we get an isotropic flag $F_\bullet = \left( \{0\} = F_0 \subsetneqq F_1\subsetneqq \cdots \subsetneqq F_k=\C^{2r}\right)$ and
a labeling $c_\bullet=(c_1>\cdots>c_k)$ with $c_j=-c_{k+1-j}$. We call $(F_\bullet, c_\bullet)$, where $c_\bullet=(c_1>\cdots>c_k)$ is a sequence with $c_j=-c_{k+1-j}$, a \emph{labeled isotropic flag}. The extended Tits building $\tilde{\mathfrak{B}}(G)$ can be realized as the collection of labeled isotropic flags.
\end{example}

A homomorphism of linear algebraic groups naturally induces a map between the corresponding extended Tits buildings. The above realization of the extended Tits building in terms of equivalence classes of one-parameter subgroups gives an easy way to construct this map.  
%\begin{remark}\label{eq:5}

\begin{definition}  \label{def-alpha-hat}
Let $\alpha: G\to G'$ be a homomorphism of linear algebraic groups. If $\lambda:\mathbb{G}_m\to G$ is a one-parameter subgroup of $G$, then $\alpha\circ\lambda$ is a one-parameter subgroup of $G'$. The map $\lambda \mapsto \alpha\circ\lambda$ respects the equivalence classes and thus gives a well-defined map $\hat{\alpha}: \tilde{\mathfrak{B}}_\Z(G)\to\tilde{\mathfrak{B}}_\Z(G')$. This extends to a map $\hat{\alpha}: \tilde{\B}(G) \to \tilde{\B}(G')$.
The map $\hat{\alpha}$ sends an extended apartment for $G$ to an extended apartment for $G'$. This is because the image of a torus in $G$ is a torus in $G'$ and every torus lies in a maximal torus. 
\end{definition}

Finally, we use the above to make the observation that the extended Tits building does not change under semidirect product with a unipotent group. In particular, the extended Tits building of a parabolic subgroup and its Levi subgroup coincide. 

\begin{prop}   \label{prop-Levi}
For a linear algebraic group $G$, suppose there exist subgroups $L$, $U \subset G$ such that $G=L\ltimes U$ (in particular, $U$ is normalized by $L$). If $U$ is unipotent then ${\tilde{\B}}(L)$ and $\tilde{\mathfrak{B}}(G)$ can be identified via the map $\hat{\iota}$ where $\iota: L \to G$ is the inclusion.  
\end{prop}

\begin{proof}
Since $L$ is a closed subgroup, it is straightforward to see that $\hat{\iota}: \tilde{\B}(L) \to \tilde{\B}(G)$ is an embedding. It remains to show $\hat{\iota}$ is surjective. Let $\gamma:\mathbb{G}_m\to G$ be a one-parameter subgroup in $G$. Since $G=L\ltimes U$, there exist a one-parameter subgroup $\gamma_L:\mathbb{G}_m\to L\simeq G/U$ and a morphism $\gamma_U: \mathbb{G}_m \to U$ such that $\gamma(s)=\gamma_L(s)\gamma_U(s),\ \forall s\in\mathbb{G}_m$.
Since the unipotent group $U$ can be embedded in $\GL(r)$ as a subvariety of upper triangular matrices with $1$'s on the diagonal, 
$\displaystyle\lim_{s\to0}\gamma_U(s)$ exists in $U$. Therefore,
    \begin{equation*}
        \lim_{s\to0}\gamma(s)\gamma_L^{-1}(s)=\lim_{s\to0}\gamma_U(s)\in U\subset G.
    \end{equation*}
    This shows $\gamma\sim\gamma_L$ and hence $\hat{\iota}$ is surjective. %{\color{red} Why is it injective?}. Hence $\hat{\iota}: \tilde{\mathfrak{B}}(L) \to \tilde{\mathfrak{B}}(G)$ is a bijection.
\end{proof}

\section{Preliminaries on toric principal bundles}
In this section we review the classification of (framed) toric principal bundles in \cite{Kaveh-Manon}. Let $T \cong \mathbb{G}_m^n$ denote an $n$-dimensional algebraic torus over an algebraically closed field $\k$. We let $M$ and $N$ denote its character and cocharacter lattices respectively. We also denote by $M_\R$ and $N_\R$ the $\R$-vector spaces spanned by $M$ and $N$. %For cone $\sigma \in N_\R$ let $M_\sigma$ be the quotient lattice
%$M_\sigma = M / (\sigma^\perp \cap M).$
Let $\Sigma$ be a (finite rational polyhedral) fan in $N_\R$ and let $X_\Sigma$ be the corresponding toric variety. Also $U_\sigma$ denotes the invariant affine open subset in $X_\Sigma$ corresponding to a cone $\sigma \in \Sigma$. We denote the support of $\Sigma$, that is the union of all the cones in $\Sigma$, by $|\Sigma|$. For each $i$, $\Sigma(i)$ denotes the subset of $i$-dimensional cones in $\Sigma$. In particular, $\Sigma(1)$ is the set of rays in $\Sigma$. For each ray $\rho \in \Sigma(1)$ we let $v_\rho$ be the primitive vector along $\rho$, i.e. $v_\rho$ is the shortest non-zero integral vector on $\rho$.

Throughout the paper we fix a point $x_0$ in the open torus orbit in $X_\Sigma$. It gives an identification of the torus $T$ with the open orbit via $t \mapsto t \cdot x_0$. 

We start by recalling the notion of a principal bundle. Let $G$ be an algebraic group, a \emph{principal $G$-bundle over a variety $X$} is a fiber bundle $\mathcal{P}$ over $X$ with an action of $G$ such that $G$ preserves each fiber and the action is free and transitive. Throughout, we take the action of $G$ on $\mathcal{P}$ to be a \emph{right} action. 

%\begin{definition}   \label{def-morphism-pb}
Let $G, G'$ be algebraic groups and $\mathcal{P}$ (respectively $\mathcal{P}'$) be a principal $G$-bundle (respectively $G'$-bundle) over $X$. A \emph{morphism of principal bundles with respect to a homomorphism of algebraic groups $\alpha:G\to G'$} is a bundle map $F:\mathcal{P}\to\mathcal{P}'$
 such that
    \begin{equation*}
        F(z\cdot g)=F(z)\cdot\alpha(g), \ \forall z\in\mathcal{P}, \forall g\in G.
    \end{equation*}
We refer to a morphism between toric principal $G$-bundles, with respect to the identity homomorphism $G \to G$, simply as a \emph{morphism of principal $G$-bundles}. We note that any morphism of principal $G$-bundles is an isomorphism.

\begin{definition}[Toric principal bundle]   \label{def-tpb}
    Let $X_\Sigma$ be the toric variety associated to a fan $\Sigma$ and $G$ an algebraic group. A \emph{toric principal $G$-bundle over $X_\Sigma$} is a principal $G$-bundle $\mathcal{P}$ together with a torus action lifting that of $X_\Sigma$,
such that the $T$-action and the $G$-action on $\mathcal{P}$ commute. More precisely, 
$\forall t\in T, \forall x\in X_\Sigma, \forall z\in\mathcal{P}_x$ we have:
    \begin{align*}
        t:\mathcal{P}_x&\to\mathcal{P}_{t\cdot x}, \\
        t\cdot (z\cdot g) &= (t\cdot z)\cdot g.
    \end{align*}
    
Recall that we have fixed a point $x_0$ in the open torus orbit in $X_\Sigma$. We call a toric principal $G$-bundle $\mathcal{P}$ together with a choice of a point $p_0\in\mathcal{P}_{x_0}$ a \emph{framed toric principal $G$-bundle}.
\end{definition}

\begin{definition}   \label{def-morphism-tpb}
A \emph{morphism of toric principal bundles}  is a morphism $F$ of principal bundles (with respect to some homomorphism $\alpha$ as above) that is also $T$-equivariant.
    A \emph{morphism of framed principal bundles} $(\mathcal{P}, p_0) \to (\mathcal{P}', p_0')$ is a morphism $F$ that sends $p_0\in\mathcal{P}_{x_0}$ to
    $p_0'\in\mathcal{P}'_{x_0}$.
\end{definition}

The following is the main combinatorial gadget to classify (framed) toric principal bundles. It can be thought of as a generalization of a real-valued piecewise linear function $\varphi: |\Sigma| \to \R$. 
\begin{definition}[Piecewise linear map]
Let $G$ be a linear algebraic group with $\tilde{\B}(G)$, the extended Tits building of $G$. Let $\Sigma$ be a fan in $N_\R$, we say that a map $\Phi:\abs{\Sigma}\to\tilde{\mathfrak{B}}(G)$ is a \emph{piecewise linear map} if:
    \begin{enumerate}
        \item[(a)] For each cone $\sigma\in\Sigma$, there exists a maximal torus $H_\sigma$ (not necessarily unique) such that $\Phi(\sigma)$ lies in an extended apartment $\tilde{A}_\sigma=\Lambda_\R^\vee(H_\sigma)$.
        \item[(b)] For each cone $\sigma\in\Sigma$, the restriction $\Phi\mid_\sigma:\sigma\to\tilde{A}_\sigma$ is an $\R$-linear map.
    \end{enumerate}
We say that a piecewise linear map $\Phi$ is \emph{integral} if $\Phi$ sends lattice points to lattice points, i.e. for any $\sigma \in \Sigma$, $\Phi(\sigma \cap N) \subset \Lambda^\vee(H_\sigma)$.
\end{definition}

\begin{definition}[Equivariant triviality]
%    Let $G$ be a linear algebraic group. Let $\mathcal{P}$ be a toric principal $G$-bundle on an affine toric variety $U_\sigma$. The product variety $U_\sigma\times G$ is a trivial principal $G$-bundle where $G$ acts on each fiber by group multiplication.
%    Let $\phi_\sigma: T\to G$ be a group homomorphism. If 
%    \begin{equation*}
%        t\cdot(x, g) = (t\cdot x, \phi_\sigma(t)g),\ \forall t\in T, \forall x\in U_\sigma, \forall g\in G,
%    \end{equation*}
%    we say $U_\sigma\times G$ is \emph{equivariantly trivial}. 
 We say that a toric principal bundle $\mathcal{P}$ on an affine toric variety $U_\sigma$ is \emph{equivariantly trivial} if there exists a toric principal $G$-bundle isomorphism between $\mathcal{P}$ and $U_\sigma\times G$, where $T$ acts on $U_\sigma \times G$ via an algebraic group homomorphism $\phi_\sigma: T\to G$ by:
\begin{equation*}
        t\cdot(x, g) = (t\cdot x, \phi_\sigma(t)g),\ \forall t\in T, \forall x\in U_\sigma, \forall g\in G.
\end{equation*}
\end{definition}

\begin{definition}[Local equivariant triviality]
Let $\mathcal{P}$ be a toric principal $G$-bundle on a toric variety $X_\Sigma$. We say that $\mathcal{P}$ is \emph{locally equivariantly trivial} if for any $\sigma \in \Sigma$, the restriction $\mathcal{P}\mid_{U_\sigma}$, to the affine open chart $U_\sigma$, is equivariantly trivial.
\end{definition}

The following gives a classification of locally equivariantly trivial framed toric principal bundles in terms of piecewise linear maps (\cite[Theorem 2.4]{Kaveh-Manon}).

\begin{theorem} \label{th-KM}
Let $G$ be a linear algebraic group over $\mathbf{k}$. 
\begin{itemize}
\item[(a)] There is a one-to-one correspondence between the isomorphism classes of locally equivariantly trivial framed toric principal $G$-bundles $\mathcal{P}$ over $X_\Sigma$ and the integral piecewise linear maps $\Phi:\abs{\Sigma}\to\tilde{\mathfrak{B}}(G)$.
    
\item[(b)] Moreover, let $\alpha:G\to G'$ be a homomorphism of linear algebraic groups. Let $(\mathcal{P}, p_0)$ (respectively $(\mathcal{P}', p'_0)$) be a locally equivariantly trivial framed toric principal $G$-bundle (respectively $G'$-bundle) with corresponding 
    piecewise linear map $\Phi:\abs{\Sigma}\to\tilde{\mathfrak{B}}(G)$ (respectively $\Phi':\abs{\Sigma}\to\tilde{\mathfrak{B}}(G')$). Then there is a (necessarily unique) morphism of framed toric principal bundles $F:\mathcal{P}\to\mathcal{P}'$ with respect to $\alpha$ if and only if $\Phi' = \hat{\alpha}\circ\Phi$.
\end{itemize}
\end{theorem}

The idea of proof of Theorem \ref{th-KM} as follows: Let $\Phi:\abs{\Sigma}\to\tilde{\mathfrak{B}}(G)$ be a piecewise linear map. For each cone $\sigma \in \Sigma$, the integral linear map $\Phi_{|\sigma}$ gives an algebraic group homomorphism $T_\sigma \to H_\sigma$ where $T_\sigma$ is the stabilizer of the orbit $O_\sigma$. Extend this to a homomorphism $\phi_\sigma:T \to H_\sigma$. On each affine chart $U_\sigma$, consider the trivial toric principal bundle $\mathcal{P}_\sigma = U_\sigma\times G$ where $T$ acts on $G$ via $\phi_\sigma$. For two cones
$\sigma$, $\sigma' \in \Sigma$ with $\tau=\sigma\cap\sigma'$, define the transition function $\psi_{\sigma, \sigma'}: U_\tau = U_\sigma \cap U_{\sigma'} \to G$ by defining it on the open orbit by $\psi_{\sigma, \sigma'} (t\cdot x_0)=\phi_{\sigma'}(t)\phi_\sigma(t)^{-1}$. One shows that this extends to a regular function $\psi_{\sigma, \sigma'}: U_\tau \to G$. The toric principal bundle $\mathcal{P}$, associated to $\Phi$, is obtained by gluing the $\mathcal{P}_{\sigma}$ via the transition functions $\phi_{\sigma, \sigma'}$. 

It is shown in \cite[Theorem 4.1]{Dey} that if $G$ is reductive then any toric principal $G$-bundle is locally equivariantly trivial. Thus Theorem \ref{th-KM} immediately implies the following.

\begin{corollary} \label{cor-KM} Let $G$ be a reductive algebraic group over $\mathbf{k}$. 
\begin{itemize}
\item[(a)] There is a one-to-one correspondence between the isomorphism classes of framed toric principal $G$-bundles $\mathcal{P}$ over $X_\Sigma$ and the integral piecewise linear maps $\Phi:\abs{\Sigma}\to\tilde{\mathfrak{B}}(G)$.
    
\item[(b)] Moreover, let $\alpha:G\to G'$ be a homomorphism of reductive algebraic groups. Let $\mathcal{P}$ (respectively $\mathcal{P}'$) be a framed toric principal $G$-bundle (respectively $G'$-bundle) with corresponding 
    piecewise linear map $\Phi:\abs{\Sigma}\to\tilde{\mathfrak{B}}(G)$ (respectively $\Phi':\abs{\Sigma}\to\tilde{\mathfrak{B}}(G')$). Then there is a morphism of framed toric principal bundles $F:\mathcal{P}\to\mathcal{P}'$ with respect to $\alpha$ if and only if $\Phi' = \hat{\alpha}\circ\Phi$.
\end{itemize}
\end{corollary}

\begin{remark}[Toric principal bundles over $\C$]   \label{rem-equiv-trivial-C}
Using analytic methods, it is also shown in \cite{Biswas} that when the base field $\k = \C$, the local equivariant triviality of toric principal bundles holds for any linear algebraic group. Hence Corollary \ref{cor-KM} also holds for linear algebraic groups over $\C$.
\end{remark}

%From the proof of \eqref{eq:6}, we have the following lemma
The following is a simple corollary of Theorem \ref{th-KM}(b).
\begin{lemma}   \label{lem-KM}
    Let $(\mathcal{P},p_0)$ be a locally equivariantly trivial framed toric principal $G$-bundle with corresponding integral piecewise linear map $\Phi:\abs{\Sigma}\to\tilde{\mathfrak{B}}(G)$. Then for any $g_0\in G$, the corresponding
integral piecewise linear map for the framed toric principal $G$-bundle $(\mathcal{P},p_0\cdot g_0)$ is $\hat{\alpha}_{g_0}\circ \Phi$, where $\alpha_{g_0}: G \to G$ is the conjugation homomorphim $x \mapsto g_0^{-1} x g_0$.
\end{lemma}
\begin{proof}
The right action by $g_0$ gives a morphism of framed toric principal bundles from $(\mathcal{P}, p_0)$ to $(\mathcal{P}, p_0 \cdot g_0)$ with respect to the conjugation homomorphism $\alpha_{g_0}: G \to G$. Theorem \ref{th-KM}(b) then implies that the piecewise linear map of $(\mathcal{P}, p_0 \cdot g_0)$ is $\hat{\alpha}_{g_0}\circ \Phi$.
\end{proof}

\section{Equivariant automorphism group}
In this section we use the classification of framed toric principal bundles (Theorem \ref{th-KM}) to give a short proof of a result of Dasgupta et al (\cite[Proposition 5.1]{Dasgupta}) describing the equivariant automorphism group of a toric principal bundle. 
\begin{definition}
%    Let $X_\Sigma$ be the toric variety associated to a fan $\Sigma$ and $G$ an algebraic group. 
Let $\mathcal{P}$ be a toric principal $G$-bundle  over a toric variety $X_\Sigma$. A \emph{$T$-equivariant automorphism} $F$ on $\mathcal{P}$ is a morphism of
principal $G$-bundles $F:\mathcal{P}\to\mathcal{P}$ which is $T$-equivariant. In other words, in the sense of Definition \ref{def-morphism-tpb}, $F$ is a morphism of toric principal bundles with respect to the identity homomorphism $\textup{id}: G \to G$. We let $\Aut_T(\mathcal{P})$ denote the group of $T$-equivariant automorphisms of $\mathcal{P}$.

%such that
%    \begin{equation}
%        F(z\cdot g)=F(z)\cdot g,\ \forall z \in \mathcal{P}, \forall g\in G.
%    \end{equation}
\end{definition}

\begin{theorem} \label{th-auto-gp}
%    Let $X_\Sigma$ be the toric variety associated to a fan $\Sigma$ and $G$ an algebraic group. 
Let $\mathcal{P}$ be a locally equivariant trivial toric principal $G$-bundle  over a toric variety $X_\Sigma$. Pick a frame $p_0 \in \mathcal{P}_{x_0}$ and let $\Phi:|\Sigma| \to \tilde{\B}(G)$ be the piecewise linear map associated to $(\mathcal{P}_{x_0}, p_0)$. We have:
    \begin{equation*}
        \Aut_T(\mathcal{P}) \cong \bigcap_{\rho\in\Sigma(1)}P_\rho,
    \end{equation*}
where $P_\rho$ is the parabolic subgroup in $G$ corresponding to $\Phi(v_\rho) \in \tilde{\B}_\Z(G)$ (see \eqref{equ-P-lambda}). %Here $\Phi$ is the corresponding integral piecewise linear map to any framed bundle $(\mathcal{P},p_0)$.
\end{theorem}
\begin{proof}
Let $F\in\Aut_T(\mathcal{P})$ and let $F(p_0)=p'_0$.
Let $\Phi$, $\Phi'$ be the piecewise linear maps corresponding to the framed bundles $(\mathcal{P},p_0),(\mathcal{P},p'_0)$ respectively. There exists a $g_0\in G$ such that $p_0'=p_0\cdot g_0$.
By Lemma \ref{lem-KM} we have $\Phi' = \hat{\alpha}_{g_0} \circ \Phi$ where $\alpha_{g_0}: G \to G$ is the conjugation by $g_0$. 
It is straightforward to check that $F \mapsto g_0$ gives an injective homomorphism $\eta: \Aut_T(\mathcal{P}) \to G$. It is injective because, firstly $F$ is determined by its values on the open orbit. Moreover, by $T$ and $G$-equivariance, $F$ is determined on the open orbit by its value at the single point $p_0$.
We need to show that the image coincides with $\bigcap_{\rho \in \Sigma(1)} P_\rho$. Note that Theorem \ref{th-KM}(b) implies that $\Phi' = \Phi$ because the automorphism $F$ is equivariant with respect to the identity $\textup{id}: G \to G$. It follows that $g_0$ is in the image of $\eta$ if and only of $\hat{\alpha}_{g_0}\circ \Phi = \Phi$. This means that, for any lattice point $x \in |\Sigma| \cap N$ we have $g_0^{-1} \Phi(x) g_0 \sim \Phi(x)$. In view of piecewise linearity of $\Phi$ this is equivalent to: $$g_0^{-1} \Phi(v_\rho) g_0 \sim \Phi(v_\rho), \quad \forall \rho \in \Sigma(1),$$
where $v_\rho$ is the shortest non-zero integral vector on $\rho$.
In view of \eqref{equ-P-lambda}, this is the case if and only of $g_0 \in \bigcap_{\rho \in \Sigma(1)} P_\rho$.
\end{proof}

\section{Equivariant reduction of structure group}
In this section we address the question of reduction of structure group for toric principal bundles.

\begin{definition}[Equivariant reduction of structure group] \label{def-eq-str-gp}
    Let $K$ be a closed subgroup of a linear algebraic group $G$. We say that a toric principal $G$-bundle $\mathcal{P}$ over $X_\Sigma$ has an 
    \emph{equivariant reduction of structure group to $K$} if there exsits a toric principal $K$-bundle $\mathcal{P}'$ over $X_\Sigma$ such that
    there is an isomorphism of toric principal $G$-bundles between $\mathcal{P}$ and $\mathcal{P}'\times^K G$, where $\mathcal{P}'\times^K G$ is the quotient of $\mathcal{P}' \times G$ by the right action of $K$ given by: 
$$(p, g) \cdot k = (pk, k^{-1}g), \quad \forall p \in \mathcal{P},~ \forall k \in K,~ \forall g \in G.$$ The group $G$ acts on $\mathcal{P}'\times^K G$ by right multiplication on the second component and with this action $\mathcal{P}'\times^K G$ is a principal $G$-bundle.
    If $\mathcal{P}$ admits an equivariant reduction of structure group to a maximal torus in $G$, then we say $\mathcal{P}$ \emph{splits equivariantly}.
\end{definition}

\begin{remark}   \label{rem-P-P'}
    Let $\iota: K \hookrightarrow G$ be the inclusion map and $F:\mathcal{P}'\to \mathcal{P}'\times^K G$ be defined by $F(p')=(p',1)$, where $1$ is the identity element in $G$. It is not difficult to see that $F$ is a morphism of principal bundles with
    respect to the homomorphism $\iota$ since
    \begin{equation*}
        F(p'\cdot k)=(p'\cdot k,1)=(p'\cdot k k^{-1},k\cdot 1)=(p', k)=(p',1)\cdot k=F(p')\cdot\iota(k).
    \end{equation*}
\end{remark}

\begin{remark}
    A toric principal $G$-bundle $\mathcal{P}$ over $X_\Sigma$ has an equivariant reduction of structure group to $K$ just means
    $\mathcal{P}$ has equivariant trivializations whose transition functions all lie in $K$.
\end{remark}

The inclusion map $\iota: K \hookrightarrow G$, gives an embedding $\hat{\iota}:\tilde{\mathfrak{B}}(K)\hookrightarrow\tilde{\mathfrak{B}}(G)$ (see Definition \ref{def-alpha-hat}). For any extended apartment $\tilde{A}_H\subset \tilde{\mathfrak{B}}(G)$,
the preimage of $\tilde{A}_H$ lies in an extended apartment in $\tilde{\mathfrak{B}}(K)$.

\begin{theorem}[Criterion for equivariant reduction of structure group]  \label{th-cri-red-st-gp}
A locally equivariantly trivial toric principal $G$-bundle $\mathcal{P}$ over $X_\Sigma$ has an equivariant reduction of structure group to $K$ if and only if there exists a $p_0\in\mathcal{P}_{x_0}$ such that the image of $\Phi$
    lies in $\tilde{\mathfrak{B}}(K)$. Here $\Phi:\abs{\Sigma}\to\tilde{\mathfrak{B}}(G)$ is the integral piecewise linear map corresponding to the framed bundle $(\mathcal{P},p_0)$.
\end{theorem}
\begin{proof}
Suppose $\mathcal{P}$ has an equivariant reduction of structure group to $K$. Then there exists a toric principal $K$-bundle $\mathcal{P}'$ over $X_\Sigma$ such that
$\mathcal{P} \simeq \mathcal{P}'\times^K G$ as toric $G$-principal bundles. 
Let $\Phi':\abs{\Sigma}\to\tilde{\mathfrak{B}}(K)$ be the corresponding integral linear map of $(\mathcal{P}', p_0')$ for some $p_0'\in\mathcal{P}'_{x_0}$.
Then $\hat{\iota}\circ\Phi':\abs{\Sigma}\to\tilde{\mathfrak{B}}(G)$ is an integral piecewise linear map as well. From Theorem \ref{th-KM}(b), we know $\hat{\iota}\circ \Phi'$ is the integral piecewise linear map
 corresponding to $(\mathcal{P}'\times^K G, (p_0', 1))$, i.e. there exists a $(p_0', 1)\in\mathcal{P}_{x_0}$ such that the image of $\hat{\iota}\circ\Phi'$ lies in $\tilde{\mathfrak{B}}(K)$. Conversely, suppose there exists a $p_0\in\mathcal{P}_{x_0}$ such that the image of $\Phi$, the integral piecewise linear map corresponding to $(\mathcal{P},p_0)$, lies in $\tilde{\mathfrak{B}}(K)$, where $\Phi:\abs{\Sigma}\to\tilde{\mathfrak{B}}(G)$.
Since the image of $\Phi$
lies in $\tilde{\mathfrak{B}}(K)$, we have a piecewise linear map $\Phi': |\Sigma| \to \tilde{\B}(K)$ such that $\Phi = \hat{\iota} \circ \Phi'$. Let $\mathcal{P}'$ be the framed toric principal bundle corresponding to $\Phi'$. As above, by Remark \ref{rem-P-P'} and Theorem \ref{th-KM}(b),  $(\mathcal{P}' \times^K G, (p_0', 1))$ is the framed toric principal $G$-bundle corresponding to $\hat{\iota} \circ \Phi'$. Therefore, $\mathcal{P} \cong \mathcal{P}' \times^K G$ as toric $G$-principal bundles.
\end{proof}

\begin{corollary}[Criterion for equivariant splitting]\label{cor-eq-split}
A locally equivariantly trivial toric principal $G$-bundle $\mathcal{P}$ over $X_\Sigma$ splits equivariantly if and only if for some (and hence any) $p_0\in\mathcal{P}_{x_0}$ the image of $\Phi$ lies in an extended apartment $\tilde{A}_H$ for some maximal torus $H\subset G$. Here $\Phi:\abs{\Sigma}\to\tilde{\B}(G)$ is the integral piecewise linear map corresponding to the framed bundle $(\mathcal{P}, p_0)$.
\end{corollary}
\begin{proof}
By definition, $\tilde{\B}(H)$ is the extended apartment $\tilde{A}_H$. The claim follows from this and Theorem \ref{th-cri-red-st-gp}.
\end{proof}

%\begin{remark}\label{eq:12}
%    If such a $p_0\in\mathcal{P}_{x_0}$ exists and let $\Phi'$ be the integral piecewise linear map corresponding to $(\mathcal{P}, p_0')$ for any other $p_0'\in\mathcal{P}_{x_0}$, then there exists a $g\in G$ such that $p_0' = p_0\cdot g$. From \eqref{lem-KM}, $\Phi' = g^{-1}\Phi g$. Therefore, the image of $\Phi'$ also lies in a
%    cone over an apartment $\tilde{A}_{g^{-1}Hg}$.
%\end{remark}

Theorem \ref{th-cri-red-st-gp} readily implies the following result of Dasgupta et al (\cite[Theorem 6.9]{Dasgupta}). 
\begin{corollary}   \label{cor-split-K-G}
Let $K$ be a closed subgroup of a linear algebraic group $G$. Let $\mathcal{P}'$ be a locally equivariantly trivial toric principal $K$-bundle over $X_\Sigma$.
    If $\mathcal{P}=\mathcal{P}'\times^K G$ splits equivariantly (as a $G$-bundle), then $\mathcal{P}'$ splits equivariantly (as a $K$-bundle).
\end{corollary}

\begin{proof}
%Let $\iota:K\to G$ be the identity map. Then $\Phi=\hat{\iota}\circ\Phi'$, where $\Phi':\abs{\Sigma}\to\tilde{\mathfrak{B}}(K)$ is the integral piecewise linear map corresponding to $(\mathcal{P}',p_0')$. Then the image of $\Phi'$ lies in a cone over an apartment in $\tilde{\mathfrak{B}}(K)$. From \eqref{eq:12}, $\mathcal{P}'$ splits equivariantly.
As before let $\iota: K \hookrightarrow G$ denote the inclusion map. We consider $\tilde{\B}(K)$ as a subset of $\tilde{\B}(G)$ via the embedding $\hat{\iota}: \tilde{\B}(K) \hookrightarrow \tilde{\B}(G)$. As explained above, for any frame $(p'_0,1) \in \mathcal{P}_{x_0}$, the image of the piecewise linear map $\Phi$ corresponding to $(\mathcal{P}' \times^K G, (p'_0, 1))$ lies in $\tilde{\B}(K)$. Since this bundle splits equivariantly, Corollary \ref{cor-eq-split} implies that this image moreover lies in $\tilde{\B}(H)$, for some maximal torus $H \subset G$. Now since the connected component of the identity in $H \cap G$ is a torus, it is contained in some maximal torus $H' \subset K$. This means that $\tilde{\B}(K) \cap \tilde{\B}(H) \subset \tilde{\B}(H')$ which, in light of Corollary \ref{cor-eq-split}, implies that $\mathcal{P}'$ also splits equivariantly. 
%From the proof of \eqref{eq:7}, there exists a $(p_0', 1_G)\in\mathcal{P}_{x_0}$ such that the image of $\Phi$ lies in $\tilde{\mathfrak{B}}(K)$, where $\Phi:\abs{\Sigma}\to\tilde{\mathfrak{B}}(G)$ is the integral piecewise linear map corresponding to $(\mathcal{P}, (p_0', 1_G))$.
%    From \eqref{eq:12}, the image of $\Phi$ lies in a cone over an apartment $\tilde{A}_H$ for some maximal torus $H\subset G$. Then the image of $\Phi$ lies in $\tilde{A}_H\cap\tilde{\mathfrak{B}}(K)$, which is also a cone over an apartment in $\tilde{\mathfrak{B}}(K)$. 
\end{proof}

In \cite[Theorem 6.1.2]{Klyachko} as well as \cite[Corollary 3.5]{Kaneyama88}, it is shown that any toric vector bundle of rank $r$ over $\mathbb{P}^n$ splits equivariantly, for $r < n$. In our language, any toric principal $\GL(r)$-bundle over $\mathbb{P}^n$ splits equivariantly, for $r < n$. As observed in \cite[Theorem 6.1]{Dasgupta}, this combined with Corollary \ref{cor-split-K-G} gives us the following.
 
\begin{corollary}\label{cor-Pn-split}
Let $K$ be a closed subgroup of $\GL(r)$. Any toric principal $K$-bundle on $\mathbb{P}^n$ splits equivariantly if $r < n$.
\end{corollary}
\begin{proof}
Let $\mathcal{P}$ be a toric principal $K$-bundle on $\mathbb{P}^n$ where $r < n$. One knows that $\mathcal{P}\times^K \GL(r)$ splits equivariantly. Then by Corollary \ref{cor-split-K-G}, $\mathcal{P}$ also splits equivariantly.
\end{proof}

%\begin{prop}
%    Let $G$ be a reductive algebraic group over $\C$. Then $G$ contains a unique maximal normal unipotent subgroup called the unipotent radical of $G$ and denoted by $R_u(G)$.
%\end{prop}

%\begin{definition}
%{\color{red} Recall that a Levi subgroup of $G$ is the centralizer of some torus in $G$.
%\end{definition}
%Recall that for any parabolic subgroup $P$ in $G$, there exists a Levi subgroup such that $P=L\ltimes R_u(P)$.}

Finally, from Theorem \ref{th-cri-red-st-gp} we obtain a short proof of \cite[Proposition 6.4]{Dasgupta} about reduction of the structure group of a toric principal $P$-bundle, where $P$ is a parabolic subgroup, to its Levi subgroup. In fact, we give a slightly more general version of this result for any linear algebraic group that can be written as a semidirect product of a subgroup and a unipotent subgroup.
  
\begin{corollary}[Equivariant reduction of structure group to a Levi]
Let $P$ be a linear algebraic group that can be written as a semidirect product $P = L \ltimes U$ of subgroups $L$ and $U$ where $U$ is unipotent. Let $\mathcal{P}$ be a locally equivariantly trivial toric principal $P$-bundle. Then $\mathcal{P}$ has an equivariant reduction of structure group to $L$. This in particular applies to the Levi decomposition $P = L \ltimes R_u(P)$ of a parabolic subgroup $P$. 
%Let $P$ be a parabolic subgroup of a reductive algebraic group $G$ over $\C$. Let $\mathcal{P}$ be a toric principal $P$-bundle over $X_\Sigma$. Then $\mathcal{P}$ has an equivariant reduction of structure group to $L$, where $L$ is the Levi subgroup in the decomposition $P=L\ltimes R_u(P)$.
\end{corollary}
\begin{proof}
From \ref{prop-Levi}, $\tilde{\mathfrak{B}}(P)\simeq\tilde{\mathfrak{B}}(L)$. By Theorem \ref{th-cri-red-st-gp}, $\mathcal{P}$ has an equivariant reduction of structure group to $L$.
\end{proof}

\begin{example}[Toric principal bundles over $\mathbb{P}^1$]
    Let $\mathcal{P}$ be a toric principal $G$-bundle over $X_\Sigma=\mathbb{P}^1$. The fan $\Sigma$ consists of two cones $\sigma_1=\left\langle 1 \right\rangle$ and $\sigma_2=\left\langle -1 \right\rangle$ in $1$-dimensional space. 
    For any $p_0\in\mathcal{P}_{x_0}$, the corresponding integral piecewise linear map $\Phi$ gives us two simplices $\Phi(\sigma_1)$ and $\Phi(\sigma_2)$. Since any two simplices lie in an apartment, there exists a maximal torus $H\subset G$
    such that $\Phi(\abs{\Sigma})\subset \tilde{A}_H$ and hence $\mathcal{P}$ splits equivariantly.
\end{example}

\begin{example}[Toric orthogonal principal bundle]
Let $\mathcal{P}$ be a toric principal $\SO(r)$-bundle. %Then $\mathcal{P}\times^{\SO(r)}\C^r$ is a toric principal $\GL(r)$-bundle and it has an equivariant reduction of structure group to $\SO(r)$.
From Corollary \ref{cor-Pn-split} it follows that any toric principal $\SO(r)$-bundle over $\mathbb{P}^n$ splits equivariantly when $r < n$. 
\end{example}

\section{Helly's number of a building}
In this section we introduce Helly's number of the Tits building of a linear algebraic group. More generally, we define Helly's number for an (abstract) building.

%{\color{red} Give the abstract definition of Helly's number for a collection of subsets of a set. State classical Helly's theorem for convex subsets in $\R^n$.}

The classical Helly's theorem in convex geometry asserts the following: let $S$ be a finite collection of convex subsets in $\R^n$ such that any $n+1$ of these convex subsets have non-empty intersection, then the intersection of all the convex sets in $S$ is non-empty.

Motivated by this theorem, one defines Helly's number for any collection of sets. Let $\mathcal{F}$ be a collection of sets. \emph{Helly's number} $h(\mathcal{F})$ of $\mathcal{F}$ is the minimal positive integer $h$ such that if a finite subcollection $S \subset \mathcal{F}$ satisfies
$\bigcap_{X \in S'}  X \neq \emptyset$ for all $S' \subset S$ with $\abs{S'}\le h$, then $\bigcap_{X \in S} X \neq\emptyset$. Helly's theorem about convex sets tells us that for the collection $\mathcal{F}$ of compact convex subsets of $\R^n$, we have $h(\mathcal{F}) \leq n+1$. In fact, it is not hard to see that $h(\mathcal{F}) = n+1$ (\cite{Helly}).

Motivated by \cite[Section 6]{Klyachko}, we give an analogous definition for the collection of parabolic subgroups of a linear algebraic group $G$. The difference with the usual notion of Helly's number is that instead of asking that a collection of parabolic subgroups have a non-empty intersection, we ask that their intersection contains a maximal torus.
\begin{definition}[Helly's number of a Tits building] \label{def-helly-gp}
Let $G$ be a linear algebraic group. We define \emph{Helly's number} $h(G)$ of $G$ to be the minimal positive integer $k$ such that the following holds: if $S$ is a collection of parabolic subgroups of $G$
such that the intersection of any $k$ elements in $S$ contains a maximal torus, then the intersection of all the elements in $S$ contains a maximal torus.
\end{definition}

\begin{remark}
It is not difficult to see that the above Helly's number is different from usual Helly's number for the collection of parabolic subgroups of $G$. That is, a finite intersection of parabolic subgroups may have non-empty intersection but does not contain a maximal torus. 
\end{remark}

More generally we define Helly's number of an abstract building.
\begin{definition}[Helly's number of a building]
Let $\Delta$ be a building. We define \emph{Helly's number} $h(\Delta)$ of $\Delta$ to be the minimal positive integer $k$ such that the following holds: if $S$ is a collection of simplices of $\Delta$
such that any $k$ simplices in $S$ lie in an apartment, then all of the simplices in $S$ lie in the same apartment.
\end{definition}

%In terms of Tits building, $h(G)$ is the minimal positive integer $k$ such that the following holds: if $S$ is a collection of simplices in $\Delta(G)$
%such that any $k$ elements in $S$ lie in an apartment, then all the elements in $S$ lie in an apartment.

%{\color{red} Problem: determine or give a upper bound for the Helly number $H(G)$ for any semisimple group $G$.}

In \cite[Section 6]{Klyachko}, Klyachko shows that $h(\GL(r))=r+1$. Therefore, for $G \hookrightarrow \text{GL}(r)$, we have $h(G) \le r + 1$. A natural question is how to find a sharp upper bound for $h(G)$ for any semisimple algebraic group $G$.
More generally, we pose the following problem:
\begin{problem}
For a building $\Delta$, give a sharp upper bound for Helly's number $h(\Delta)$.
\end{problem}

%{\color{red}
%State the statement about splitting of tpbs and Helly's number as a corollary.} 

From Corollary \ref{cor-eq-split}, we have the following corollary.
\begin{corollary}\label{cor-proj-split}
Let $G$ be a reductive algebraic group. Then any toric principal $G$-bundle on $\mathbb{P}^k$ splits equivariantly when $k\ge h(G)$.
\end{corollary}

\begin{proof}
Let $(\mathcal{P},p_0)$ be a framed toric principal $G$-bundle over $\mathbb{P}^k$. Let $\Phi:\abs{\Sigma}\to\tilde{\mathfrak{B}}(G)$ be the integral piecewise linear map corresponding to $(\mathcal{P},p_0)$ where $\Sigma$ is the fan of $\mathbb{P}^k$. In the fan $\Sigma$, there are $k+1$ rays and any collection of $k$ rays lies in some maximal cone $\sigma$. Since $\Phi(\sigma)$ lies in an extended apartment $\tilde{A}_\sigma$, we see that the images of any collection of $k$ rays lies in an extended apartment. Since $k\ge h(G)$, the image of any $h(G)$ rays also lies in an extended apartment. By the definition of $h(G)$, we then conclude that the images of all the $k+1$ rays of $\Sigma$ belong to the same apartment. Now Corollary \ref{cor-eq-split}, implies that $\mathcal{P}$ splits equivariantly.
\end{proof}
   
%\begin{proof}
%Follows from Klyachko. Need the obvious statement that intersection of a maximal torus in $GL(m)$ with $G$ is contained in a maximal torus in $G$. 
%\end{proof}

%{\color{red} Examples: state any computation or partial result about Helly's number of $\text{Sp}(2)$ and/or $\text{Sp}(2r)$.}
\begin{example}
    Let $G=\text{Sp}(2)$. Since $\text{Sp}(2)\subset \text{GL}(2)$, $h(G)\le 2+1=3$. Consider three isotropic flags
    \begin{equation*}
        \begin{split}
            F_1&=(\{0\}\subsetneqq\{e_1\}\subsetneqq \C^2)\\
            F_2&=(\{0\}\subsetneqq\{f_1\}\subsetneqq \C^2)\\
            F_3&=(\{0\}\subsetneqq\{e_1+f_1\}\subsetneqq \C^2),
        \end{split}
    \end{equation*}
    where $\{e_1,f_1\}$ is a normal basis of $\C^2$.
    Any $2$ of these flags are adapted to a normal frame, but all of them are not adapted to any normal frame. This shows $h(G)>2$. Therefore, $h(G)=3$.
\end{example}

\section{Appendix}   \label{sec-appendix}
For the sake of completeness in this appendix we give the defining axioms of an (abstract) building.
\begin{definition}[Building]    \label{def-bldg}
A building is a pair $(\Delta,\mathcal{A})$ consisting of a simplicial complex $\Delta$ and a family $\mathcal{A}$ of subcomplexes $A$ (\emph{apartments}) satisfying the following conditions:
\begin{enumerate}
        \item each simplex of $\Delta$ or any apartment $A$ is contained in a maximal simplex (\emph{chamber}), and each chamber of $\Delta$ or $A$ has the same finite dimension $n$;
        \item each apartment $A$ is connected, in the sense that for any two chambers $C, D$ in $A$ there is a sequence of chambers of $A$ starting with $C$ and ending with $D$, the intersection of any two successive members of which is an $(n-1)$-simplex;
        \item any $(n-1)$-simplex of $\Delta$ (respectively, of any apartment $A$) is contained in more than 2 chambers of $\Delta$ (respectively, in exactly 2 chambers of $A$);
        \item any two chambers $C, D$ of $\Delta$ are contained in some apartment;
        \item if two simplices $C, C'$ of $\Delta$ are contained in two apartments $A, A'$, then there is an isomorphism from $A$ onto $A'$ fixing both $C$ and $C'$ pointwise.
    \end{enumerate}
\end{definition}

Extending the construction of Tits building of a linear algebraic group as the collection of its parabolic subgroups, there is a group theoretic way to construct buildings using the notion of a \emph{Tits system} or a \emph{$(B, N)$ pair}. A Tits system is a structure on groups of Lie type and roughly speaking says that such groups have structure similar to that of the general linear group over a field. 
\begin{definition}[Tits system]
A \emph{Tits system} or \emph{$(B, N)$ pair} is a collection $(G,B,N,S)$, where $B$ and $N$ are subgroups of a group $G$ and $S$ is a subset of $N/(B\cap N)$ satisfying the following conditions:
    \begin{enumerate}
        \item $H=B\cap N$ generates $G$;
        \item $H\lhd N$;
        \item $S$ generates $W=N/H$ and consists of elements of order 2;
        \item $sBw\subset BwB\cup BswB$, $\forall s\in S, w\in W$;
        \item $sBs\not\subset B$, $\forall s\in S$.
    \end{enumerate}
\end{definition}

A subgroup of $G$ is called \emph{parabolic} if it contains a conjugate of $B$. The collection of all parabolic subgroups in a Tits system can be given the structure of a building (\cite[Section 6.2]{Abramenko-Brown}).

\end{document}